\documentclass[11pt,reqno]{amsart}
\usepackage{xifthen}
\usepackage{subfig}
\usepackage{pkgfile}
\usepackage[T1]{fontenc}
\usepackage[utf8]{inputenc}
\usepackage{authblk}

\newcommand{\bbV}{{\mathcal V}}
\newcommand{\bbI}{{\mathcal I}}

\newcommand{\remove}[1]{}
\newtheorem{conj}{Conjecture}[section]
\newtheorem{fact}{Fact}[section]

\newtheorem{prop}{Proposition}[section]

\usepackage{lipsum}
\makeatletter
\g@addto@macro{\endabstract}{\@setabstract}
\newcommand{\authorfootnotes}{\renewcommand\thefootnote{\@fnsymbol\c@footnote}}%
\makeatother

\begin{document}
\begin{center}
\Large
Almost Empty Monochromatic Triangles in Planar Point Sets \par \bigskip

 \normalsize
Deepan Basu\textsuperscript{1}, Kinjal Basu\textsuperscript{2}, Bhaswar B. Bhattacharya\textsuperscript{2} and Sandip Das\textsuperscript{3} \par \bigskip

\textsuperscript{1}\small{Max-Planck Institut f\" ur Mathematik in den Naturwissenschaften, Leipzig, Germany,\\ {\tt Deepan.Basu@mis.mpg.de} \par
\textsuperscript{2}Department of Statistics, Stanford University, California, USA\\ {\tt \{kinjal, bhaswar\}@stanford.edu} \par
\textsuperscript{3}Advanced Computing and Microelectronic Unit, Indian Statistical Institute, Kolkata\\
 {\tt sandipdas@isical.ac.in}  \par }\bigskip

\end{center}

\begin{abstract}
For positive integers $c, s \geq 1$, let $M_3(c, s)$ be the least integer such that any set of at least
$M_3(c, s)$ points in the plane, no three on a line and colored with $c$ colors, contains a monochromatic triangle with at most $s$ interior points. The case $s=0$, which corresponds to empty monochromatic triangles, has been studied extensively over the last few years. In particular, it is known that $M_3(1, 0)=3$, $M_3(2, 0)=9$ and $M_3(c, 0)=\infty$, for $c\geq 3$. In this paper we extend these results when $c \geq 2$ and $s \geq 1$. We prove that the least integer $\lambda_3(c)$ such that $M_3(c, \lambda_3(c))< \infty$ satisfies:
$$\left\lfloor\frac{c-1}{2}\right\rfloor \leq\lambda_3(c)\leq c-2,$$
where $c \geq 2$. Moreover, the exact values of $M_3(c, s)$ are determined for small values of $c$ and $s$. We also conjecture that $\lambda_3(4)=1$, and verify it for sufficiently large Horton sets. \\

\end{abstract}

\section{Introduction}
\label{sec:intro}

The Erd\H os-Szekeres theorem \cite{erdos} states that for every positive integer {\it m}, there exists a smallest integer $ES(m)$, such that any set of at least $ES(m)$ points in the plane, no three on a line, contains $m$ points which lie on the vertices of a convex polygon. The best known bounds on $ES(m)$ are:
\begin{equation}
2^{m-2} +1\leq ES(m) \leq  {{2m-5}\choose {m-3}}+1.
\label{eq:es_bound}
\end{equation}
The lower bound is due to Erd\H os \cite{erdosz} and the upper bound is due to T\'oth and Valtr \cite{tothvaltr}.
It is known that $ES(4)=5$ and $ES(5)=9$ \cite{kalb}. Following a long computer search, Szekeres and Peters
\cite{szekeres} proved that $ES(6)=17$. The exact value of $ES(m)$ is unknown for all $m> 6$, but Erd\H os conjectured that the lower bound in (\ref{eq:es_bound}) is, in fact, sharp.

In 1978 Erd\H os \cite{erdosempty} asked whether for every positive integer $k$,
there exists a smallest integer $H(k)$, such that any set of at least
$H(k)$ points in the plane, no three on a line, contains $k$ points which lie on the vertices of a
convex polygon whose interior contains no points of the set. Such a subset is called an {\it empty
convex $k$-gon} or a {\it convex k-hole}. Esther Klein showed that {$ H(4)=5$} and Harborth \cite{harborth}
proved that {$ H(5)=10$}. Horton \cite{horton} showed
that it is possible to construct an arbitrarily large set of points
without a 7-hole, thereby proving that {$ H(k) $} does not exist
for {$ k \geq 7$}. After a long wait, the existence of $ H(6)$ was proved
by Gerken \cite{gerken} and independently by Nicol\'as
\cite{nicolas}. Later, Valtr \cite{valtrhexagon} gave a simpler version of Gerken's proof. Recently, Koshelev \cite{koshelev} has proved that $H(6) \leq \max\{ES(8), 400\}\leq 463$ (the full version of the paper appeared in Russian \cite{koshelev_russian}). The counting version of this problem, that is, the minimum number $N_n(k)$ of
$k$-holes in a set of $n$ points in general position in the plane, is also widely studied \cite{barany_valtr_empty_convex_polygons,dumitrescu} (for the best known bounds refer to the recent papers by Aichholzer et al. \cite{aichholzerI,aichholzerII} and Valtr \cite{valtr_count} and the references therein).

The notion of {\it almost empty} convex polygons, that is, convex polygons with few interior points, was introduced by Nyklov\'a \cite{nyklova}. For integers $k\geq 3$ and $s\geq 0$, let $H(k, s)$ be the minimum positive integer such that any set $S$ of at least $H(k, s)$ points in the plane, no three on a line, contains a subset $Z$ of cardinality $k$ whose elements are on the vertices of a convex $k$-gon and there are at most $s$ points of $S$ in the interior of the convex hull of $Z$. It is clear that $ES(k) \leq H(k, s) \leq H(k, 0)= H(k)$. Nyklov\'a \cite{nyklova} showed that $H(6, 6) = ES(6)=17$ and $H(6, 5) = 19$. It was also shown that $H(k, s)$ does not exist for every $s \leq 2^{(k+6)/4}-k-3$  \cite{nyklova}. Recently, Koshelev \cite{koshelev_almost_empty_hexagon} proved that $H(6, 1)\leq ES(7)$. Koshelev  \cite{koshelev_computer_solution,koshelev_almost_empty_hexagon} also showed that Nyklov\'a's proof is incorrect, and using a computer search based on the Szekeres-McKay-Peters algorithm \cite{szekeres} proved that $H(6, k)=17$, for all $2\leq k\leq 6$ and $H(6, 1)=18$.\footnote{Koshelev's papers were originally published in Russian. Here we refer to their English translations \cite{koshelev_computer_solution,koshelev_almost_empty_hexagon}.}

Colored variants of the Erd\H os-Szekeres problem were considered by Devillers et al. \cite{erdosszekereschromatic}.
In such problems, the set of points is partitioned into $r\geq 2$ color classes, and a convex polygon is said to be monochromatic if all its vertices belong to the same color class. It is easy to see that any 2-colored point set of size 10 has an empty monochromatic triangle. Grima et al. \cite{disjoint_monochromatic_triangles} showed that 9 points are necessary and sufficient for a bi-colored point set to have a monochromatic empty triangle. Aichholzer et al. \cite{aichholzer_triangle_number} proved that any set of $n$ bi-colored points in the plane contains $\mathrm\Omega(n^{5/4})$ empty monochromatic triangles, which was later improved to $\mathrm\Omega(n^{4/3})$ by Pach and T\'oth \cite{pach_toth_monochromatic_empty_triangles}. Devillers et al. \cite{erdosszekereschromatic} also constructed a specific coloring for the Horton sets to prove that there exists arbitrarily large 3-colored point sets with no monochromatic empty triangles. Using this they showed that there exists arbitrarily large 2-colored point sets with no monochromatic 5-hole.

It was conjectured by Devillers et al.\cite{erdosszekereschromatic} that any sufficiently large set of bi-colored points in the plane, no three on a line, contains a monochromatic 4-hole. However, in spite of substantial attempts over the last few years, the problem remains open.
Devillers et al.\cite{erdosszekereschromatic} showed that for $n\geq 64$ any bi-chromatic Horton set contains a monochromatic 4-hole. The best known lower bound is a two-colored set of 36 points that contains no monochromatic 4-hole by Huemer and Seara \cite{huemer_geombinatorics}.

As no real progress for the monochromatic empty 4-hole problem was being made, a weaker version arose by relaxing the convexity condition \cite{pach_non_convex}. Aichholzer et al. \cite{aichholzer_siam_dm} showed that this relaxed conjecture is indeed true. They proved that if the cardinality of the bi-chromatic point set $S$ is at least 5044, there always exists an empty monochromatic 4-gon spanned by $S$, that need not be convex. Using observations on vertex degree parity constraints for triangulations of $S$, this bound was lowered to 2760 points by Aichholzer et al. \cite{aiccholzer_parity_wads}.


To the best of our knowledge, the analogous version of the almost empty convex polygon problem for colored point sets has never been considered. In this paper, we initiate the study of the combinatorial quantities associated with the existence of monochromatic empty polygons with few interior points in colored point sets.
For positive integers $c, s \geq 1$ and $r\geq 3$ define $M_r(c, s)$ to be the least integer such that any set of at least $M_r(c, s)$ points in the plane, no three on a line, and colored with $c$ colors contains a monochromatic convex polygon of $r$ vertices and with at most $s$ interior points. In this paper, we primarily concentrate on almost empty monochromatic triangles, that is, the case $r=3$. The case $s=0$, which corresponds to empty monochromatic triangles, has been studied extensively over the last few years. We already mentioned that $M_3(1, 0)=3$, $M_3(2, 0)=9$ \cite{disjoint_monochromatic_triangles} and $M_3(c, 0)=\infty$, for $c\geq 3$ \cite{erdosszekereschromatic}. In this paper we extend these results when $c \geq 3$ and $s \geq 1$. We begin by showing that 

\begin{thm}
$M_3(3, 1)=13$. 
\label{th:31}
\end{thm}


Using this result and induction on $c$, and a Horton set construction, we obtain the main result of the paper:

\begin{thm}
For $c\geq 2$, the least integer $\lambda_3(c)$ such that $M_3(c, \lambda_3(c))< \infty$ satisfies:
$$\left\lfloor\frac{c-1}{2}\right\rfloor\leq\lambda_3(c)\leq c-2.$$
\label{th:lambda}
\end{thm}

To prove the lower bound in Theorem \ref{th:lambda} we use colored Horton sets. The bounds obtained in this way are equal for $c\in \{2, 3\}$.  For $c=4$, the first value for which the bounds in Theorem \ref{th:lambda} do not match, we show that

\begin{thm}
Any Horton set with at least $26$ points colored arbitrarily with $4$ colors contains a monochromatic triangle with at most $1$ interior point.
\label{th:41_restriction}
\end{thm}

Due to the above theorem and several other observations we conjecture that the lower bound in Theorem \ref{th:lambda} is also tight for $c=4$, that is, every sufficiently large 4-colored point set contains a monochromatic triangle with at most 1 interior point.

\section{Preliminaries}

We begin by introducing the definitions and notation required for the remainder of
the paper. Let $S$ be a finite set of points in the plane in general position,
that is, no three on a line. Denote the convex hull of $S$ by $CH(S)$. The
boundary vertices of $CH(S)$, and the interior of $CH(S)$
are denoted by $\bbV(CH(S))$ and $\bbI (CH(S))$, respectively. Finally, for any finite set $Z\subseteq S$, $|Z|$ denotes the cardinality of $Z$.

A set of points $S$ in general position is said to be $c$-{\it colored} if $S$ can be partitioned into $c$ non-empty sets $S = S_1\bigsqcup S_2 \bigsqcup \cdots \bigsqcup S_c$, where each set $S_i$ will be referred to as the set of points of color $i$, and $\bigsqcup$ denotes the union of disjoint sets. A subset $T\subset S$ is called {\it monochromatic} if all its points have the same color.

A triangulation of a set of points $S$ in the plane in general position is a triangulation of the convex hull of $S$, with all points from $S$ being among the vertices of the triangulation. We now recall a few standard facts from triangulations of point-sets that we shall use repeatedly in the subsequent sections.

\begin{fact}
For any set $S$ of points in the plane in general position, the number of triangles in any triangulation of $S$ is $2|S|-|\bbV(CH(S))|-2$.
\label{fact1}
\end{fact}

Since for point sets in general position $|\bbV(CH(S))|\leq |S|$, the following is an immediate consequence of the above fact:

\begin{fact}
For any set $S$ of points in the plane in general position, the number of triangles in any triangulation of $S$ is at least $|S|-2$.
\label{fact2}
\end{fact}

Grima et al. \cite{disjoint_monochromatic_triangles} proved that every 2-colored 9-point set contains a monochromatic empty triangle. Moreover, there exists a set of 8 points with 4 points colored red and 4 points colored blue and no empty monochromatic triangle. In the following lemma we show that this is the only possible distribution of 2 colors among 8 points that can lead to a point set with no monochromatic empty triangles.

\begin{lem}\label{lm2}
Let $S=R\bigsqcup B$ be a 2-colored point set with color class $R$ and $B$, with $|S|=8$ and $|R|\neq |B|$. Then $S$ always contains an empty monochromatic triangle.
\end{lem}

\begin{proof} Without loss of generality we assume that $|R|>|B|$, so $|R|\geq 5$.  If $|R|\ge 6$, then by Fact \ref{fact2} we get at least four triangles by triangulating $R$. Since $|B|\le 2$ we obtain an empty triangle in $R$. If $|R|=5$ and $|\mathcal V(CH(R))|\le 4$, then we get at least four triangles in $R$ by Fact \ref{fact1}. In this case we have $|B|=3$, so we again obtain a monochromatic empty triangle in $R$. If $|\mathcal V(CH(R))|=5$, then there is an empty triangle in $R$ unless $\bbI(CH(R))\cap S=B$. In that case, the 3 points in $B$ form an empty triangle. 
\end{proof}

\section{Monochromatic Triangles with At Most 1 Interior Point: Few Exact Values}
\label{sec:31}

In this section we obtain the exact values of $M_3(c, 1)$ for $c \leq 3$. These results will be used later on in the proofs of the main theorem, and are also interesting in their own right.

\begin{prop}
$M_3(2,1)=6$.
\label{lm1}
\end{prop}

\begin{proof} We observe that $M_3(2,1)>5$, by arbitrarily placing 2 blue-colored points inside a red-colored triangle.


Consider a set of 6 points colored arbitrarily with 2 colors, red and blue. Let $R$ and $B$ denote the set of red and blue points respectively. Without loss of generality assume that $|R|\ge |B|$. If $|R|\geq 4$, by Fact \ref{fact2} we can partition $R$ in at least two interior disjoint triangles. In this case, $|B|\leq 2$ and we have a monochromatic triangle in $R$ with at most one interior point. If $|R|=|B|=3$, then there is a monochromatic triangle with at most one interior point whenever $|\bbI(CH(R))\cap S|\ne 2$. Otherwise, $|\bbI(CH(R))\cap S|=2$ and the monochromatic triangle spanned by $B$ has at most one interior point. 
\end{proof}


\begin{figure}[!hbt]
\centering
\centering
 \includegraphics[width=0.45\columnwidth]
 {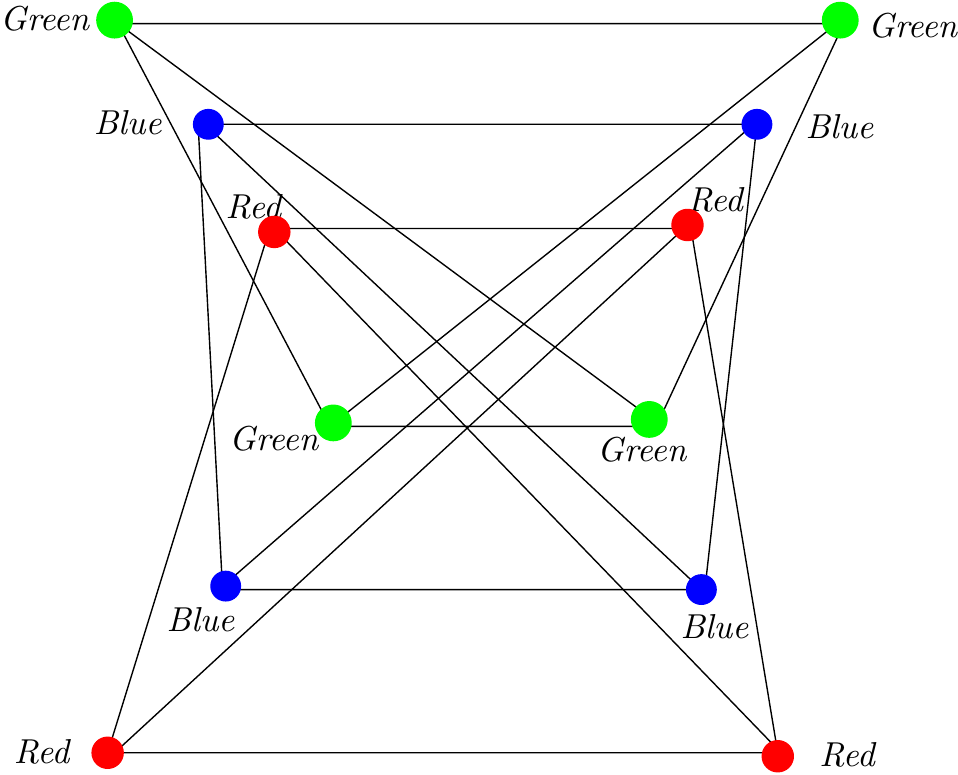}\\
\caption{A set of 3-colored 12 points that does not contain any monochromatic triangle with at most 1 interior point.}
\label{fig1}
\end{figure}

\subsection{Proof of Theorem \ref{th:31}}

Figure \ref{fig1} shows that $M_3(3,1)\ge 13$. Therefore, to prove the theorem it suffices to show that $M_3 (3,1)\le 13$.

Let $S=R\bigsqcup B \bigsqcup G$ be a any set of 13 points colored by 3 colors. The three color classes $R, B, G$, will be referred to as red, blue, and green, respectively. Assume, without loss of generality,  that $|R|\ge |B| \ge |G|$. As $|R|+|B|+|G|=13$, we have $|R|\ge 5$. We now have the following cases:

\begin{description}
\item[$|R|\ge 6$] In this case any triangulation of $R$ has at least 4 red triangles. If all of these 4 triangles have more than one interior point then $|S\backslash R| \geq 4 \times 2 =8$, which contradicts the fact $|S\backslash R|=13-6=7$. Hence we get a red triangle with at most one interior point.

\item[$|R|=5$] We consider three cases based on the size of $\bbV(CH(R))$:

\begin{description}
\item[$|\bbV(CH(R))|=5$] Fact \ref{fact1} implies that that any triangulation of $R$ has three triangles. Assume that all these three triangles have more than one interior point. Then $Z=\bbI(CH(R))\cap S$ is a 2-colored point set with $|Z|\geq 3\times 2=6$, and we get a monochromatic empty triangle with at most one interior point by Proposition \ref{lm1}.

\item[$|\bbV(CH(R))|=4$] Fact \ref{fact1} implies that that any triangulation of $R$ has four triangles. As $|S\backslash R|=8$, we have a monochromatic red triangle with at most one interior point unless $\bbI(CH(R))\cap S=S\backslash R=B\bigsqcup G$. Now, If $|B|\ne |G|$, we then get a monochromatic empty triangle by Lemma \ref{lm2}. Note that this triangle is empty in $B\bigsqcup G$. But, it can contain at most 1 red point in the interior, since $|\bbI(CH(R))\cap R|=1$, and we are done. Finally, if $|B|=|G|=4$, then we look at $CH(B)$. If $|\bbV(CH(B))|=3$, we have 3 interior disjoint blue triangles by Fact \ref{fact1}. In this case $|S\backslash (\bbV(CH(R))\cup B)|= |\bbI(CH(R))\cap R|+|G|=5 < 3\times 2 =6$, and so we get a blue triangle with at most one interior point. Otherwise, $|\bbV(CH(B))|=4$. In this case, we are done unless $|\bbI(CH(B))\cap S| \geq 4$. Since $|\bbI(CH(R))\cap R|=1$, three of the 4 points in $\bbI(CH(B))\cap S$ must be green, and we get a green triangle with at most one red point inside.

\item[$|\bbV(CH(R))|=3$] In this case any triangulation of $R$ has 5 triangles. As $|S\backslash R| =8 < 5 \times 2 =10$, we always have a red triangle with at most one interior point. \vspace{0.05in}
\end{description}
\end{description}

This completes the proof of Theorem \ref{th:31}, where the exact value of $M_3(3,1)$ is obtained. In the following section we show that $M_3(5, 1)=\infty$. It remains open to determine whether $M_3(4, 1)$ is finite. More discussions about $M_3(4, 1)$ are in Section \ref{sec:41}.

\section{Proof of Theorem \ref{th:lambda}}
\label{sec:lambda}

In this section bounds on $\lambda_3(c)$ as stated in Theorem \ref{th:lambda} will be proved. Note that $\lambda_3(c)$ is the least integer such that $M_3(c, \lambda_3(c)) < \infty$. Proofs of the upper and lower bounds are presented separately in two sections. First it is shown that $M_3(c, c-2)\leq c^2+c+1$ for $c\geq 3$, that is, any $c$-colored point set with cardinality at least $c^2+c+1$ contains a monochromatic triangle with at most $c-2$ interior points. This proves that $\lambda_3(c) \leq c-2$. In the next section, constructing colorings for the Horton set, the lower bound on $\lambda_3(c)$ is proved.

\subsection{Proof of Upper Bound}

We begin with the following lemma where an upper bound on $M_3(c, c-1)$ is obtained. This bound will then be used to bound $M_3(c, c-2)$. 

\begin{lem}
For $c\geq 2$, $M_3(c,c-1) \leq \max\{c^2 + 1, 6\}$
\label{lm:2}
\end{lem}

\begin{proof} Let the $c$ colors be indexed by $\{1, 2, \ldots, c\}$ and $S$ be a any set of  $c^2+1$  points in the plane colored by $c$ colors. The result will be proved by induction on $c$. The base case $c=2$ is true as $M_3(2,1)=6$ by Proposition \ref{lm1}.

Now, assume that $c\geq 3$ and $M_3(c-1, c-2)\leq \max\{(c-1)^2+1, 6\}$. Let $S_1 \subseteq S$ be the set of points with color 1. As $|S|=c^2+1$, by the pigeon-hole principle, it can be assumed that $|S_1|\geq c+1$.  

If $|S_1|\geq c+2$, by Fact \ref{fact2} any triangulation of $S_1$ has at least $c$ triangles. There exists a monochromatic triangle of color 1 with at most $c-1$ interior points unless each of the triangles in any triangulation of $S_1$ contains at least $c$ interior points. This implies that $|S|\geq c^2+|S_1|\geq c^2+c+2$, which is impossible. 

Therefore, $|S_1|=c+1$. In this case any triangulation of $S_1$ has at least $c-1$ triangles. If any triangulation of $S_1$ has $c$ triangles, then by similar arguments as before there exists a monochromatic triangle with at most $c-1$ interior points. This implies that any triangulation of $S_1$ contains exactly $c-1$ triangles, that is, $|\bbV(CH(S_1))|=c+1$. Let $S_{-1}=(S\backslash S_1)\cap CH(S_1)$. Note that if $|S_{-1}| \leq  c(c-1)-1$, then there exists a monochromatic triangle of color 1, with at most $c-1$ interior points. Hence, it suffices to assume that $|S_{-1}|\geq c(c-1)\geq M_3(c-1, c-2)$. As $|S_{-1}|$ is a set of points colored with $c-1$ colors, by induction hypothesis there exists monochromatic triangle in $|S_{-1}|$ with at most $c-2$ interior points.  
\end{proof}

\subsubsection{Completing the Proof of the Upper Bound} 
For $c=2$, it is known that $M_3(2, 0)=9$ \cite{disjoint_monochromatic_triangles}, which implies that $\lambda_3(2)=0$. For $c\geq 3$, it will be shown by induction on $c$ that $M_3(c, c-2)< c^2 + c + 1$. The base case $c=3$ is true as $M_3(3,1) = 13$ by Theorem \ref{th:31}. 

Let $c\geq 4$ and $S$ be any set of  $c^2+c+1$  points in the plane colored by $c$ colors, and assume that the theorem is true for $c-1$, that is, $M_3(c-1,c-3) \leq c^2 - c + 1$. Let $S_1 \subseteq S$ be the set of points with color 1. As $|S|=c^2+c+1$, by the pigeon-hole principle, it can be assumed that $|S_1|\geq c+2$.  

If $|S_1|\geq c+3$, by Fact \ref{fact2} any triangulation of $S_1$ has at least $c+1$ triangles. There exists a monochromatic triangle with at most $c-2$ interior points unless each of the triangles in any triangulation of $S_1$ contains at least $c-1$ interior points. This implies that $|S|\geq (c-1)(c+1)+|S_1|\geq c^2+c+2$, which is impossible. 

Therefore, $|S_1|=c+2$. If $|\bbV(CH(S_1))|\leq c$, then any triangulation of $S_1$ has at least $c+2$ triangles. By similar arguments as before, there exists a monochromatic triangle with at most $c-2$ interior points. Therefore, it suffices to assume that $c+1\leq |\bbV(CH(S_1))|\leq c+2$. These two cases are considered separately as follows:

\begin{description}

\item[$|\bbV(CH(S_1))|=c+2$]  Any triangulation of $S_1$ contains $c$ triangles. Let $S_{-1}=(S\backslash S_1)\cap CH(S_1)$. Note that if $|S_{-1}|<c(c-1)$, then there exists a monochromatic triangle of color 1, with at most $c-2$ interior points. Hence, it suffices to assume that $|S_{-1}|\geq c(c-1)$. As $S_{-1}$ is a set of $c(c-1)$ points colored with $c-1$ colors, and $M_3(c-1, c-2)\leq (c-1)^2+1$ by Lemma \ref{lm:2} there exists a monochromatic triangle in $S_{-1}$ with at most $c-2$ interior points.

\item[$|\bbV(CH(S_1))|=c+1$] Any triangulation of $S_1$ contains $c+1$ triangles by Fact \ref{fact1}. As before, let $S_{-1}=(S\backslash S_1)\cap CH(S_1)$. Note that if $|S_{-1}|<(c+1)(c-1)$, then there exists a monochromatic triangle of color 1, with at most $c-2$ interior points. Hence, it suffices to assume that $|S_{-1}|= c^2-1$. As $S_{-1}$ has $c^2 -1$ points colored with $c-1$ colors, and by induction hypothesis $M_3(c-1, c-3)\leq c^2-c+1$, there exists a monochromatic triangle $\Delta$ in $S_{-1}$ with at most $c-3$ interior points of $S_{-1}$. As the convex hull of $S_1$ contains only one point of color 1, $\Delta$ can contain at most one point of color 1. This implies that $\Delta$ is monochromatic triangle with at most $c-2$ interior points. 

\end{description}

\subsection{ Proof of Lower Bound} 

In this section we prove the lower bound on $\lambda_3(c)$ by appropriately coloring a Horton set with $c$ colors. A {\it Horton set} is a set $H$ of $n$ points sorted by $x$-coordinates: $h_1 <_x h_2 <_x h_3 <_x \ldots <_x h_n$, such that the set
$H^+:= \{h_2,h_4, \ldots\}$ of the even points and the set $H^- := \{h_1,h_3, \ldots\}$ of the odd points are Horton sets and any line through two even points (the {\it upper set}) leaves all odd points below and any line through two odd points (the {\it lower set}) leaves all even points above. A Horton set of size $n$ can be recursively obtained by adding a large vertical separation after intertwining in the $x$-direction an upper Horton set $H^+$ of size $\lfloor n/2 \rfloor $ and a lower set $H^-$  of size $\lceil n /2\rceil$ \cite{horton}.


With these definitions the following observation can be proved easily.

\begin{prop}
Let $H=\{h_1, h_2, \ldots, h_n\}$ be a Horton set, sorted in the increasing order of the $x$-coordinates. For any line joining $h_{i}$ and $h_{j}$, define
$ H_{ij} $ as all points of $H$ with $x$-coordinates between that of $h_{i}$ and $h_{j}$. Also we define $H_{ij}^+:=H_{ij} \cap H^+$ and $H_{ij}^-:=H_{ij} \cap H^-$. If $i \in H^+$ and $j \in H^-$ (or vice versa), then the line joining $h_{i} \in H^+$ and $h_{j}$ has all points of $H^+_{ij}$ above it and all points of $H^-_{ij}$ below it.
\label{lm6}
\end{prop}

\begin{proof}
Let $h_k \in H_{ij}$ be a point which is below the line joining $h_{i} $ and $h_{j}$. If $h_k \in H^+_{ij}$, then the line through $h_{i} \in H^+$ and $h_k \in H^+$ leaves $h_{j} \in H^-$ above, which is not the case for the Horton set since any $2$ points of $H^+$, when joined by a line must have all points of $H^-$ below it. Hence, $h_k \in H^-_{ij}$. 

Similarly, it can be shown that all points which are above the line joining $h_{i} $ and $h_{j}$ are in $H^+_{ij}$.
\end{proof}


\begin{figure}[!hbt]
\centering
\centering
 \includegraphics[width=0.48\columnwidth]
 {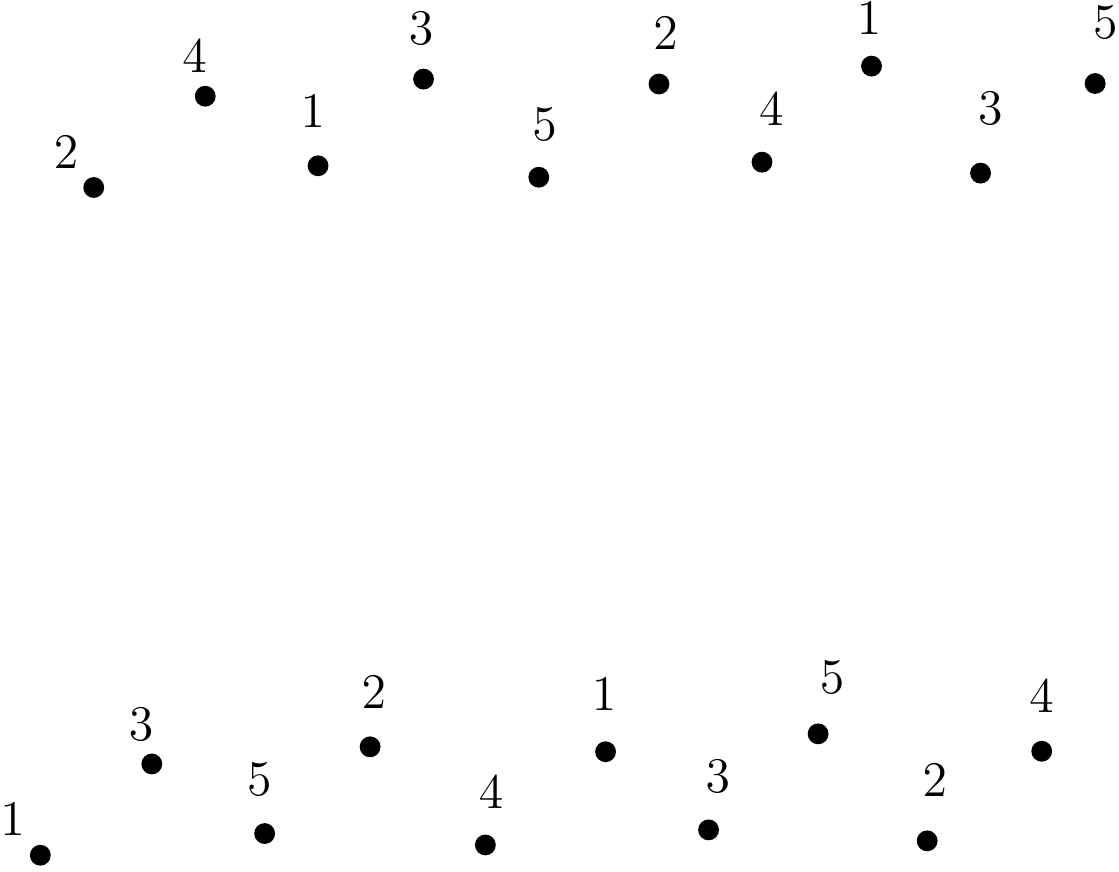}\\
\caption{A Horton set colored with 5 colors. The number next to a point denotes the color of the point. Observe that the coloring splits in a cyclical manner in $H^+$ and $H^-$.}
\label{fig2}
\end{figure}

Using the above lemma we now complete the proof of the lower bound. Let $c = 2q + 1$ be an odd number. The $c$ colors are indexed by $\{1, 2, \ldots, c\}$. Consider a Horton set $H$ of size $n$ and arrange the points $h_1,h_2, \ldots h_{n}$ of $H$ in increasing order of $x$-coordinates. Color $H$ with these $c$ colors as follows:

The points $h_1, h_{2q+2}, h_{4q+3} \ldots $ are colored by color 1. $h_2, h_{2q+3}, h_{4q+4} \ldots$ are colored by color 2, and in general, the points $h_i, h_{2q+i+1}, h_{4q+i+2}, \ldots$ by color $i$, for $i \in \{1, 2, \ldots, c\}$. As $c=2q+1$ is odd, this coloring splits in a similar pattern in $H^+$ and $H^-$, that is, $H^-$ is colored in cyclical manner as $1,3,5,\ldots ,2q+1,2,4, \ldots 2q, \ldots$ and so on, and $H^+$ is colored in cyclical manner as $2,4, \ldots 2q,1,3,5,\ldots,2q+1, \ldots$ and so on. The coloring scheme with $5$ colors is illustrated in Figure \ref{fig2}.

If $c=2q+2$ is an even number, and a Horton set $H$ of size $n$ is given, then color the rightmost point of $H$ with color $2q+2$. The remaining $n-1$ points are colored with $c-1$ colors as before. 


Hereafter, we assume that $c$ is odd. The proof for the case $c$ is even can be done similarly. Let $\Delta=\{h_a,h_b,h_c\}$ be a monochromatic triangle in a Horton set $H$, colored with $2q+1$ colors as above. From the inductive definition of a Horton set, w.l.o.g. it suffices to assume $h_b, h_c\in H^-$ and $h_a \in H^+$.  The coloring is done in such a way that the points in $H^+$ have all the $2q+1$ colors. Let $H_{bc},H_{bc}^-$ and $H_{bc}^+$ be defined as before. 

By our coloring scheme and using the fact that $h_b, h_c\in H^-$, we get $|H_{bc}|=4\gamma q+2\gamma-1$ and $|H_{bc}^-|=2\gamma q+\gamma-1$ for some $\gamma\in \{1, 2, \ldots, \}$. By recursively breaking $H^-$ into smaller Horton sets, after some steps $h_b$ and $h_c$ must belong to the upper set and lower set, respectively, of some subset of $H^-$ (say $H_0$) which is itself another Horton set. 
So by Proposition \ref{lm6} if $h_b \in {H_0}^+ $ and $h_c \in {H_0}^-$ the line $h_b h_c$ will have at least $q$ points of ${H_0}^+ \subset H^-$ above it. Moreover, these $q$ points, since they are in $H^-$, must either lie beneath the line $h_a h_b$ or beneath the line $h_a h_c$. Therefore, the monochromatic triangle $h_a h_b h_c$ must have at least $q$ interior points. This proves that any monochromatic triangle in $H$ must have at least $q$ interior points.

Therefore, we have shown that it is possible to construct arbitrarily large sets of points which when colored with $c$-colors, has no monochromatic triangle with less than $q$ interior points. This implies that $M_3(2q+2, q-1)=M_3(2q+1, q-1)= \infty$ and hence $\lambda_3(c) \geq \lfloor \frac{c-1}{2} \rfloor $. This completes the proof of Theorem \ref{th:lambda}.

\section{Proof of Theorem \ref{th:41_restriction}}
\label{sec:41}

%

In this section we prove that any sufficiently large Horton set colored arbitrarily with 4 colors contains a monochromatic triangle with at most 1 interior point. To this end, let $H=\{h_1, h_2, \ldots, h_{26}\}$ be a $4$ colored Horton set of size $26$. If all the 4 colors are not present in $H^+$ (respectively $H^-$), then by Theorem \ref{th:31}  there exists a monochromatic triangle with at most 1 interior point in $H^+$ (respectively $H^-$). Therefore, assume that all the 4-colors are present in both $H^+$ and $H^-$. W.l.o.g. let $h_a<_x h_b$ be two points in $H^-$ of same color, such that the colors of all points of $H^-$ in between $h_a$ and $h_b$ are different from that color. We can choose that color such that $b-a \in \{2,4,6\}$. 

\begin{description}
\item[$b-a=2$] In this case the points $h_a$ and $h_b$ together with any point $h_c\in H^+$ of the same color (say color 1) is an empty monochromatic triangle because, from the definition of a Horton set, no point in $H^+$ or $H^-$ can be inside the triangle $h_a h_b h_c$.

\item[$b-a=4$] By Proposition \ref{lm6} there is at most 1 point of $H^-$ above the line joining $h_a$ and $h_b$. Therefore, triangle $h_a h_b h_c$, where $h_c \in H^+$ has color 1, has at most 1 interior point.

\item[$b-a=6$] W.l.o.g assume that $h_a \in (H^-)^-, h_b \in (H^-)^+$.  Let the $2$ points in $H^-$ between $h_a$ and $h_b$ be $h_u <_x h_v$. Naturally, $h_u \in (H^-)^+$ and $h_v \in (H^-)^-$. By Proposition \ref{lm6} the only point over the line joining $h_a$ and $h_b$ is $h_u$, and $h_a h_b h_c$, where $h_c \in H^+$ has color 1, has at most 1 interior point.

\end{description}

This completes the proof of Theorem \ref{th:41_restriction} and strengthens evidence for the following conjecture:  

\begin{conj}
Every sufficiently large set of points in the plane, in general position, colored arbitrarily with $4$ colors contains a monochromatic triangle with at most $1$ interior point.
\label{conjecture1}
\end{conj}

\section{Conclusions}

In this paper, we study the existence of monochromatic empty triangles with few interior points. We prove that any large enough $c$-colored point set contains a monochromatic triangle with at most $c-2$ interior points. Using a Horton set construction, we also show that for any $c\geq 2$, there exist arbitrarily large $c$-colored point sets in which every monochromatic triangle contains at least $\left\lfloor \frac{c-1}{2} \right\rfloor $ interior points.

This paper just scratches the surface of the almost empty monochromatic polygon problem. Several interesting questions are left open. Improving the bounds on $\lambda_3(c)$ and proving Conjecture \ref{conjecture1} are the main challenges. Generalizing these results to monochromatic convex $r$-gons would also be interesting.\\

\small{\noindent{\it Acknowledgement}: The authors are indebted to the anonymous referees for carefully reading the manuscript and providing many valuable comments which improved the presentation of the paper.}

\end{document}